\def\draftdate{\today}

\documentclass{amsart}
\usepackage{amssymb}
\usepackage{xy}

\newcommand{\ssdot}{\bullet}%{{\scriptscriptstyle\bullet}}
\newcommand{\seqdot}{\bullet,\dotsc,\bullet}
\newcommand{\supdot}{^\ssdot}
\newcommand{\subdot}{_\ssdot}

\newcommand{\splus}{\Sigma^{\infty}_{+}}
\newcommand{\thick}[2][{S}]{\aT_{#2}(#1)}
\newcommand{\model}[2][{S}]{\aM_{#2}(#1)}
\newcommand{\opmodel}[2][{S}]{\aM^{\op}_{#2}(#1)}
\newcommand{\baC}{{\overline\aC}}
\newcommand{\baM}{{\overline\aM}}

\newcommand{\SPdual}{SP^{\vee}}

\let\overto\xrightarrow

\newcommand{\Spdot}[1][\ssdot]{S'_{#1}}
\newcommand{\Sdot}[1][\ssdot]{S_{#1}}

\newcommand{\Sdotmac}[2]{S^{(#1)}_{#2}}
\newcommand{\Sdotq}[1][\seqdot]{\Sdotmac{q}{#1}}

\def\fixdepth{\vrule height1em width0pt depth1ex}
\let\iso\cong
\let\sma\wedge

\renewcommand{\to}{\mathchoice{\longrightarrow}{\rightarrow}{\rightarrow}{\rightarrow}}
\newcommand{\from}{\mathchoice{\longleftarrow}{\leftarrow}{\leftarrow}{\leftarrow}}

\let\catsymbfont\mathcal

\newcommand{\aB}{{\catsymbfont{B}}}
\newcommand{\aC}{{\catsymbfont{C}}}
\newcommand{\aD}{{\catsymbfont{D}}}

\newcommand{\aM}{{\catsymbfont{M}}}

\newcommand{\aQ}{{\catsymbfont{Q}}}

\newcommand{\aT}{{\catsymbfont{T}}}

\newcommand{\bZ}{{\mathbb{Z}}}

\def\quickop#1{\expandafter\DeclareMathOperator\csname
#1\endcsname{#1}}
\quickop{id}\quickop{Id}\quickop{colim}\quickop{hocolim}
\quickop{Ar}\quickop{sing}\quickop{Hom}\quickop{Ho}
\quickop{ob}\quickop{diag}
\quickop{Ext}\quickop{Ob}\quickop{Tor}
\quickop{End}\quickop{Tot}
\quickop{Mult}

\newcommand{\w}{\mathrm{w}}

\newcommand{\op}{\mathrm{op}}

\numberwithin{equation}{section}

\newtheorem{thm}[equation]{Theorem}
\newtheorem*{thm*}{Theorem}
\newtheorem{cor}[equation]{Corollary}
\newtheorem{lem}[equation]{Lemma}
\newtheorem{prop}[equation]{Proposition}
\theoremstyle{definition}
\newtheorem{defn}[equation]{Definition}
\newtheorem{cons}[equation]{Construction}
\newtheorem{notn}[equation]{Notation}
\theoremstyle{remark}

\xyoption{arrow}
\xyoption{matrix}
\xyoption{cmtip}
\SelectTips{cm}{}
\newcommand{\term}[1]{\textit{#1}}

\newdir{ >}{{}*!/-5pt/\dir{>}}

\begin{document}

\title[Derived Koszul Duality and Involutions]
{Derived Koszul Duality and Involutions in the Algebraic $K$-Theory of Spaces}

\author{Andrew J. Blumberg}
\address{Department of Mathematics, The University of Texas,
Austin, TX \ 78712}
\email{blumberg@math.utexas.edu}
\thanks{The first author was supported in part by NSF grant DMS-0906105}
\author{Michael A. Mandell}
\address{Department of Mathematics, Indiana University,
Bloomington, IN \ 47405}
\email{mmandell@indiana.edu}
\thanks{The second author was supported in part by NSF grant DMS-0804272}

\date{\draftdate}
\subjclass[2000]{Primary 19D10; Secondary 18F25}

\begin{abstract}
We interpret different constructions of the algebraic $K$-theory of
spaces as an instance of derived Koszul (or bar) duality and also as an
instance of Morita equivalence. We relate the interplay between these
two descriptions to the homotopy involution.  We define a geometric
analog of the Swan theory $G^{\bZ}(\bZ[\pi])$ in terms of $\splus
\Omega X$ and show that it is the algebraic $K$-theory of the
$E_{\infty}$ ring spectrum $DX=S^{X_{+}}$.
\end{abstract}

\maketitle

%%%%%%%%%%%%%%%%%%%%%%%%%%%%%%%%%%%%%%%%
\section{Introduction}

Associated to a space $X$ are two ring spectra, $\splus \Omega X$, the
free suspension spectrum on the based loop space of $X$, and
$DX=S^{X_{+}}$, the Spanier-Whitehead dual of $X$.
Waldhausen~\cite{Wald} defined 
the algebraic $K$-theory of $X$, $A(X)$, as the $K$-theory of the
ring spectrum $\splus \Omega X$.  This theory has deep geometric
content: when $X$ is a manifold, $A(X)$ contains the stable
pseudo-isotopy theory of $X$, and when $X$ is a finite complex, $A(X)$
is a receptacle for ``higher torsion invariants''
\cite{DwyerWeissWilliams} and closely related to transfers
\cite{KleinWilliams}.  On the other hand, the unstable homotopy theory
of $X$ is encoded in the $E_{\infty}$ ring spectrum $DX$
\cite{Mandell}.  Recent work of Morava~\cite{Morava} conjectures the
structure and properties for a category of homotopy theoretic motives
in terms of the stabilization of a category of correspondences; one
candidate construction put forward is built from algebraic $K$-theory
of ring spectra of the form $DX$.

Based on computations in $THH$ motivated by string topology, Ralph
Cohen conjectured a duality between $K(DX)$ and
$A(X)$ as modules over $K(S)$ \cite{CohenVoronov}.  Although the
non-connectivity of $DX$ means that trace methods fail to apply to
$K(DX)$, in this paper we construct such a duality in terms of derived
Koszul duality when $X$ is a simply connected finite CW complex.

In differential graded algebra, derived Koszul duality (or bar
duality) concerns the contravariant adjunction between the category of
augmented differential graded algebras and itself
\cite{Moore-Nice,HMS} (named for the special case of Koszul algebras
\cite{Priddy, BGS, BGSo}).
The dual of an augmented differential graded $k$-algebra 
$A$ is an augmented differential graded $k$-algebra $E$ that models
the $A$-module endomorphisms of $k$, $\End_{A}(k,k)$.  Under mild
hypotheses, $A\simeq \End_{E}(k,k)$; the contravariant functors
$\Hom_{A}(-,k)$ and $\Hom_{E}(-,k)$ form an adjunction on the module
categories and an equivalence between various thick subcategories of
the derived categories.

In our context, $\splus \Omega X$ forms an augmented $S$-algebra, and
we can identify the augmented $S$-algebra of $\splus\Omega
X$-endomorphisms of $S$ as the augmented $S$-algebra $DX$
\cite[\S4.22]{DwyerGreenleesIyengar}.  In fact, the coherent
$\splus\Omega X$-module 
equivalence $S\sma S\simeq S$ makes the endomorphism ring spectrum
naturally commutative, compatibly with the natural commutative
$S$-algebra structure on $DX$.  Interpreting $\Ext_{\splus \Omega
X}(-,S)$ and $\Ext_{DX}(-,S)$ as contravariant adjoints on derived
categories
\[
\xymatrix{%
\aD_{\splus \Omega X}\ar[r]<.75ex>\fixdepth
&\aD_{DX}\mathstrut\ar[l]<.75ex>\fixdepth,
}%
\]
we get equivalences upon restricting to certain subcategories.  For
example, the subcategory $\aD^{c}_{\splus \Omega X}$ of compact
objects of $\aD_{\splus \Omega X}$ is the thick subcategory generated
by $\splus \Omega X$ and is equivalent under this adjunction to the
thick subcategory $\thick{DX}$ of $\aD_{DX}$ generated by $S$.
This is reminiscent of Waldhausen's comparison of the stable category
of $\Omega X$-spaces with the stable category of retractive spaces
over $X$.   Writing $\model{DX}$ for the subcategory of the model
category of $DX$-modules that are isomorphic in $\aD_{DX}$ to objects
in $\thick{DX}$, we prove the following theorem.  In the following
theorem and all theorems in this section, we understand
$X$ to be a simply connected finite CW-complex. 

\begin{thm}\label{maina}
$K(\model{DX})$ is weakly equivalent to $A(X)=K(\splus \Omega X)$.
\end{thm}

On the other hand, the subcategory $\aD^{c}_{DX}$ of compact objects
of $\aD_{DX}$ is the thick subcategory generated by $DX$ and is
equivalent under the adjunction above to the thick subcategory
$\thick{\splus\Omega X}$ of $\aD_{\splus \Omega X}$ generated by
$S$.  The category $\thick{\splus \Omega X}$ is a geometric
analogue of the category of finite rank projective $\pi_{1}X$-modules,
whose $K$-theory $G^{\bZ}_{i}(\bZ[\pi])$ was studied by
Swan~\cite{Swan}.  We call the 
$K$-theory of the category $\model{\splus\Omega X}$ the geometric
Swan theory of the space $X$ and denote it as $G(X)$.  We prove the
following theorem. 

\begin{thm}\label{maing}
$G(X)=K(\model{\splus \Omega X})$ is weakly equivalent to $K(DX)$.
\end{thm}

In fact both $G(X)$ and $K(DX)$ are commutative
ring spectra, and the equivalence is a weak equivalence of ring
spectra.  Likewise $A(X)$ is a module spectrum over $G(X)$ and
$K(\model{DX})$ is a module spectrum over $K(DX)$; the
equivalence in Theorem~\ref{maina} is a weak equivalence of module
spectra.  In fact, we have the following more precise result.

\begin{thm}\label{mainc}
The weak equivalence $G(X)\to K(DX)$ is a map of $E_{\infty}$ ring
spectra.  The weak equivalence $K(\model{DX}) \to A(X)$ is a map of
$G(X)$-modules.
\end{thm}

Because $X$ is a finite CW complex, the
$\splus\Omega X$-module $S$ is compact, and so we can interpret the
map on $K$-theory induced by inclusion of the thick subcategory
generated by $S$ into the thick subcategory of compact $\splus \Omega
X$-modules in terms of Waldhausen's fibration theorem.  We
obtain a localization sequence of $K$-theory spectra
\begin{equation}\label{eqloc}
G(X)\to A(X)\to K(\aC_{\splus \Omega X}/\epsilon)
\end{equation}
where $\aC_{\splus \Omega X}/\epsilon$ is the Waldhausen category of
compact $\splus\Omega X$-modules but with weak equivalences the maps
whose cofiber is in $\thick{\splus\Omega X}$.  (We typically do not
have a corresponding transfer $A(X)\to G(X)$ because $S$ is
not usually a compact $DX$-module when $X$ is a finite complex.)  We
intend to study this sequence further in a future paper.
%
% Does this (both compact) ever occur in nontrivial circustances? No
%

Derived Koszul duality between categories of modules over differential
graded algebras is a contravariant phenomenon, but there is also an
associated covariant Morita adjunction switching chirality from left
modules to right modules \cite{DwyerGreenleesIyengar}.  (For a survey
on Morita theory in stable homotopy theory, see
\cite{Shipley-Morita}.) In the 
presence of an anti-involution (for example, commutativity), we can
use the anti-involution to obtain a Morita adjunction between
categories of left modules.  In the context of $DX$ and $\splus \Omega
X$, we get two covariant adjunctions
\[
\xymatrix{%
\aD_{\splus \Omega X}\ar[r]<.75ex>\fixdepth
&\aD_{DX}\ar[l]<.75ex>\fixdepth
}%
\]
given by the adjoint pairs
\[
\Ext_{\splus \Omega X}(S,-),\Tor^{DX}(-,S)
\qquad \text{and}\qquad  \Tor^{\splus \Omega X}(-,S), \Ext_{DX}(S,-).
\]
The first restricts to an equivalence
\[
\thick{\splus \Omega X}\simeq \aD^{c}_{DX}
\]
(in fact, since $S$ is compact as an $\splus \Omega X$ module, the
adjunction restricts to embed $\aD_{DX}$ as the localizing
subcategory of $S$ in $\aD_{\splus \Omega X}$).
The second restricts to an equivalence
\[
\aD^{c}_{\splus \Omega X}\simeq \thick{DX}.
\]
These equivalences give rise to equivalences on algebraic $K$-theory,
akin to the equivalences of Theorems~\ref{maina} and ~\ref{maing}.
The composites are self-maps on $A(X)$ and $G(X)$.  In
Section~\ref{secinvolution}, we identify these self-maps as the
standard homotopy involutions.

Experts will recognize that Theorems~\ref{maina} and~\ref{maing} fit
into the framework of \cite{DwyerGreenleesIyengar}
and \cite{ToenVezzosi}; the benefit of the approach here is the
description in terms of concrete models, which allow more direct
comparisons than in the abstract approach, and more precise results
such as Theorem~\ref{mainc}.

Readers may also wonder about the connection to the work of Goresky,
Kottwitz, and MacPherson on Koszul duality \cite{GKM}.  From our
perspective, they study the ``dual'' setting in which $G = \Omega BG$
is compact and $BG$ is infinite.  Our techniques apply to recover
(integral) liftings of their equivalences of derived categories; in
fact, this case was studied in \cite{HMS}.  Because this example is
not connected as closely to $A$-theory, we have chosen to omit a
detailed discussion. 

The paper is organized as follows.  In Section~\ref{secalg} we review
and slightly extend the passage from algebraic structures on
Waldhausen categories to algebraic structures on $K$-theory spectra,
using the technology developed in \cite{ElmendorfMandell}.  In
Section~\ref{secmodels} we introduce the concrete models for $\splus
\Omega X$ and the endomorphism ring spectra, which allow a good
point-set model for the adjunctions in the remaining
sections. Section~\ref{seccontra} studies the contravariant adjunction
and proves
Theorems~\ref{maina}, \ref{maing}, and~\ref{mainc}.  Finally,
Section~\ref{secinvolution} studies the point-set model of the
covariant adjunctions of \cite{DwyerGreenleesIyengar} and identifies
the composite homotopy endomorphisms on $A(X)$ and $G(X)$ as the
standard homotopy involutions.

The authors would like to thank Ralph Cohen and Bruce Williams for
asking motivating questions, as well as Haynes Miller, John Klein,
Jack Morava, and John Rognes for helpful conversations.

%%%%%%%%%%%%%%%%%%%%%%%%%%%%%%%%%%%%%%%%%%%%%%%%%%%%%%%%%%%%%%%%%%%%%%%%
\section{Algebraic structures on {W}aldhausen $K$-theory}\label{secalg}

Since even before the advent of the theory of symmetric spectra
\cite{HSS}, experts have understood that any algebraic structure on a
Waldhausen category induces an analogous structure on Waldhausen
$K$-theory.  Sources for results of this type in the literature
include \cite[p.~342]{Wald}, \cite[App.~A]{GeisserHesselholt},
\cite{ElmendorfMandell}. We briefly review the current state of the
theory here. 

We refer the reader to \cite[\S1.2]{Wald} for the definition of a
Waldhausen category (called there a ``category with cofibrations and
weak equivalences'').  Recall that Waldhausen's $\Sdot$
construction \cite[\S1.3]{Wald} produces a simplicial Waldhausen
category $\Sdot\aC$ from a Waldhausen category $\aC$ and is defined as
follows.  Let $\Ar[n]$ denote the category with objects $(i,j)$ for
$0\leq i\leq j\leq n$ and a unique map $(i,j)\to (i',j')$ for $i\leq
i'$ and $j\leq j'$.  $\Sdot[n]\aC$ is defined to be the full
subcategory of the category of functors $A\colon \Ar[n]\to \aC$ such
that:
\begin{itemize}
\item $A_{i,i}=*$ for all $i$, 
\item The map $A_{i,j}\to A_{i,k}$ is a cofibration for all $i \leq
j \leq k$, and
\item The diagram
\[  \xymatrix@-1pc{%
A_{i,j}\ar[r]\ar[d]&A_{i,k}\ar[d]\\A_{j,j}\ar[r]&A_{j,k} } \] is a
pushout square for all $i \leq j \leq k$,
\end{itemize}
where we write $A_{i,j}$ for $A(i,j)$.  The last two conditions can be
simplified to the hypothesis that each map $A_{0,j}\to A_{0,j+1}$ is a
cofibration and the induced maps $A_{0,j}/A_{0,i}\to A_{i,j}$ are
isomorphisms.  This becomes a Waldhausen category by defining a map
$A\to B$ to be a weak equivalence when each $A_{i,j}\to B_{i,j}$ is a
weak equivalence in $\aC$, and to be a cofibration when each
$A_{i,j}\to B_{i,j}$ and each induced map
$A_{i,k}\cup_{A_{i,j}}B_{i,j}\to B_{i,k}$ is a cofibration in $\aC$.

Since $\Sdot\aC$ forms a simplicial Waldhausen category, the
construction can be iterated to form $\Sdot \Sdot \ldots \Sdot \aC$.
For our purposes, it is convenient to have an ``all at once''
construction of the $q$-th iterate $\Sdotq\aC$.  For this
construction, we need the following terminology 
(see also \cite[\S~2]{Rognes}).

\begin{defn}\label{defcofcube}
Let $[n]$ denote the ordered set $0\leq 1 \leq \dotsb \leq n$.  For a
Waldhausen category $\aC$, a functor $C\colon [n_{1}]\times \dotsb
\times [n_{q}]\to \aC$ is \term{cubically cofibrant} means that:
\begin{enumerate}
\item Every map $C(i_{1},\dotsc,i_{q})\to C(j_{1},\dotsc,j_{q})$ is cofibration,
\item in every sub-square ($1\leq r<s\leq q$)
\[
\xymatrix{%
C(i_{1},\dotsc,i_{q})\ar[r]\ar[d]&C(i_{1},\dotsc,i_{r}+1,\dotsc,i_{q})\ar[d]\\
C(i_{1},\dotsc,i_{s}+1,\dotsc,i_{q})\ar[r]&
C(i_{1},\dotsc,i_{r}+1,\dotsc,i_{s}+1,\dotsc,i_{q}),
}
\]
the induced map from the pushout to the lower-right entry
\begin{multline*}
C(i_{1},\dotsc,i_{r}+1,\dotsc,i_{q})
\cup_{C(i_{1},\dotsc,i_{q})}
C(i_{1},\dotsc,i_{s}+1,\dotsc,i_{q})\to\\
C(i_{1},\dotsc,i_{r}+1,\dotsc,i_{s}+1,\dotsc,i_{q})
\end{multline*}
is a cofibration,
\item and in general, in every $m$-dimensional sub-cube specified by
choosing $m$ distinct coordinates $1 \leq r_1 < r_2 < \ldots <
r_m \leq n$, the induced map from the colimit over the diagram
obtained by deleting $Q = C(i_1,\ldots,i_{r_1}+1,\ldots,i_{r_2}+1,\ldots,
i_{r_m}+1, \ldots, i_n)$ to $Q$ is a cofibration.
\end{enumerate}
\end{defn}

\begin{cons}[Iterated $\Sdot$ Construction]\label{consitsdot}
Let $\Ar[n_{1},\dotsc,n_{q}]$ denote the category 
$\Ar[n_{1}]\times \dotsb \times \Ar[n_{q}]$.  For a functor
\[
A\colon \Ar[n_{1},\dotsc,n_{q}]=\Ar[n_{1}] \times \dotsb \times
\Ar[n_{q}]\to \aC, 
\]
we write $A_{i_{1},j_{1};\dotsc;i_{q},j_{q}}$ for the value of $A$ on
the object 
$((i_{1},j_{1}),\dotsc,(i_{q},j_{q}))$. 
For a Waldhausen category $\aC$, let
$\Sdotq[n_{1},\dotsc,n_{q}]\aC$ be the full subcategory of functors $A$
(as above) such that: 
\begin{itemize}
\item $A_{i_{1},j_{1};\dotsc;i_{q},j_{q}}=*$ whenever $i_{k}=j_{k}$ for some $k$.
\item The subfunctor 
\[
C(j_{1},\dotsc,j_{q})=A_{0,j_{1};\dotsb;0,j_{q}}\colon  [n_{1}]\times
\dotsb \times [n_{q}]\to \aC 
\]
is cubically cofibrant.
\item For every object 
$(i_{1},j_{1};\dotsc;i_{q},j_{q})$ in $\Ar[n_{1}]\times \dotsb \times
\Ar[n_{q}]$, every $1\leq r\leq q$, and every $j_{r}\leq k\leq n_{r}$, the square
\[  \xymatrix@-1pc{%
A_{i_{1},j_{1};\dotsc;i_{q},j_{q}}\ar[r]\ar[d]
&A_{i_{1},j_{1};\dotsc;i_{r},k;\dotsc;i_{q},j_{q}}\ar[d]\\
A_{i_{1},j_{1};\dotsc;j_{r},j_{r};\dotsc;i_{q},i_{q}}\ar[r]
&A_{i_{1},j_{1};\dotsc;j_{r},k;\dotsc;i_{q},i_{q}}
} \]
is a pushout square.
\end{itemize}
The subcategory $\w\Sdotq[n_{1},\dotsc,n_{q}] \aC$ consists of the maps in
$\Sdot[n_{1},\dotsc,n_{q}] \aC$ that are objectwise weak equivalences.
We understand $\Sdotmac{0}{} \aC$ to be $\aC$ and we see that $\Sdotmac{1}{n} \aC$
is $\Sdot[n] \aC$.
\end{cons}

Following Waldhausen~\cite[p.~330]{Wald}, we define the $K$-theory
spectrum of a Waldhausen category $\aC$ to be the spectrum with $q$-th
space
\[
K\aC(q)=N(\w \Sdotq \aC)=|N\subdot (\w\Sdotq \aC)|,
\]
the geometric realization of the nerve of the
multi-simplicial category $\w\Sdotq \aC$. The 
suspension maps $\Sigma K\aC(q)\to K(q+1)$ are induced on diagrams by
the projection map
\[
\Ar[n_{1}]\times \dotsb \times \Ar[n_{q}]\times \Ar[n_{q+1}]\to
\Ar[n_{1}]\times \dotsb \times \Ar[n_{q}].
\]
Defining an action of $\Sigma_{q}$ on $K\aC(q)$ by permuting the
simplicial directions, we see from the explicit description of
$\Sdotq \aC$ above, that $K\aC$ forms a symmetric
spectrum.

We can encode an algebraic structure on a set of symmetric spectra
using a \term{symmetric multicategory} (also called \term{colored
operad}).  A symmetric multicategory $\aM$ enriched in (small) categories
consists of: 
\begin{itemize}
\item A set of \term{objects} $\Ob\aM$.
\item A (small) category of \term{$k$-morphisms}
$\aM_{k}(x_{1},\dotsc,x_{k};y)$ for all $k=0,1,2,\dotsc$ and all
$x_{1},\dotsc,x_{k},y\in \Ob\aM$.
\item A unit object $1_{x}$ in $\aM_{1}(x;x)$ for each $x\in \Ob\aM$
\item For every permutation $\sigma \in \Sigma_{k}$, an isomorphism
\[
\sigma^{*}\colon \aM_{k}(x_{1},\dotsc,x_{k};y)\to 
\aM_{k}(x_{\sigma 1},\dotsc ,x_{\sigma k}),
\]
compatibly assembling to
an action of $\Sigma_{k}$ on $\coprod \aM_{k}(x_{1},\dotsc,x_{k};y)$.
\item Composition maps
\begin{multline*}
\aM_{n}(y_{1},\dotsc,y_{n};z)\times 
(\aM_{j_{1}}(x_{1,1},\dotsc,x_{1,j_{1}};y_{1})
\times \dotsb\times 
\aM_{j_{n}}(x_{n,1},\dotsc,x_{n,j_{n}};y_{n}))\\
\to \aM_{j}(x_{1,1},\dotsc,x_{n,j_{n}};z)
\end{multline*}
\end{itemize} 
satisfying the analogue of the usual conditions for an
operad~\cite[pp.1--2]{MayGILS}; these are written out in
\cite[\S2]{ElmendorfMandell}. The following definition is standard: 

\begin{defn}
Let $\aM$ be a symmetric multicategory enriched in small categories. An
$\aM$-algebra $A$ in symmetric 
spectra consists of a symmetric spectrum $A(x)$ for each $x\in \Ob\aM$
and maps of symmetric spectra
\[
N(\aM_{k}(x_{1},\dotsc,x_{k};y))\sma A(x_{1})\sma\dotsb \sma
A(x_{k})\to A(y)
\]
for all $k$, $x_{1},\dotsc,x_{k},y$, which are compatible with the
composition maps and identity objects of $\aM$.  Here (as above) $N(-)$ denotes
the geometric realization of the nerve of the category.  When $k=0$,
we understand the map pictured above as $N(\aM(;y))\sma S\to A(y)$. 
\end{defn}

To define an $\aM$-algebra in Waldhausen categories, we first need to
describe the kinds of functors objects of $\aM_{k}$ should map to.

\begin{defn}
Let $\aC_{1},\dotsc,\aC_{n}$ and $\aD$ be Waldhausen categories.  A
functor 
\[
F\colon \aC_{1}\times \dotsb \times \aC_{n}\to \aD
\]
is \term{multiexact} if it satisfies the following conditions:
\begin{itemize}
\item $F(X_{1},\dotsc,X_{n})=*$ if any of $X_{1},\dotsc,X_{n}$ is $*$.
\item $F$ is exact in each variable
(preserves weak equivalences, cofibrations, and pushouts over
cofibrations in each variable, keeping the other variables fixed).
\item Given cofibrations $X_{k,0}\to X_{k,1}$ in $\aC_{k}$ for all
$k$, the diagram 
\[
A(i_{1},\dotsc,i_{n})=F(X_{1,i_{1}},\dotsc,X_{n,i_{n}})\colon
[1]\times \dotsb \times [1]\to \aD
\]
is cubically cofibrant.  
\end{itemize}
We define the category of multiexact functors
\[
\Mult_{n}(\aC_{1},\dotsc,\aC_{n};\aD)
\]
to have objects the multiexact
functors and maps the natural weak equivalences.  For $n=0$, we define
$\Mult_{0}(;D)$ to be $\w\aD$, the subcategory of weak equivalences in
$\aD$. 
\end{defn}

Because multiexact functors compose into multiexact functors, the
definition above makes the category of small Waldhausen 
categories into a symmetric multicategory enriched in categories.  Following
\cite{ElmendorfMandell}, we define an $\aM$-algebra in Waldhausen
categories as a map of symmetric multicategories enriched in categories.

\begin{defn}
Let $\aM$ be a symmetric multicategory enriched in small
categories. An $\aM$-algebra $C$ in 
Waldhausen categories consists of a Waldhausen category $C(x)$ for
each $x\in \Ob\aM$ and functors
\[
\aM_{k}(x_{1},\dotsc,x_{k};y)\to \Mult_{k}(\aC(x_{1}),\dotsc,\aC(x_{k}); \aD)
\]
for all $k$, $x_{1},\dotsc,x_{k},y$, which are compatible with the
permutations, composition maps, and identity objects of $\aM$.
\end{defn}

Recalling the universal property of the smash product of symmetric
spectra \cite[2.1.4]{HSS}, the following theorem is immediate from
inspection of the definitions above.

\begin{thm}\label{thm:multstruct}
Waldhausen's algebraic $K$-theory functor naturally takes
$\aM$-al\-ge\-bras in Waldhausen categories to $\aM$-algebras in
symmetric spectra. 
\end{thm}

In particular, as explained in \cite[\S9]{ElmendorfMandell}, the
preceding theorem applies to describe the algebraic structures on
$K$-theory spectra induced by pairings on the level of Waldhausen
categories.  Suppose that $\aC$ is Waldhausen category which is
also a permutative category, where the product $\otimes \colon \aC
\times \aC \to \aC$ is a biexact functor; we will refer to $\aC$ as a
\term{permutative Waldhausen category}.  Recall 
that a permutative category is a rigidified form of a symmetric
monoidal category: a permutative category is a symmetric monoidal
category where the product satisfies strict associativity and unit
relations (the associativity and unit isomorphisms are the
identity). If
$\aC$ is a permutative Waldhausen category, a \term{strict Waldhausen module}
over $\aC$ consists of a Waldhausen category $\aQ$ and a biexact
functor $\aC \times \aQ \to \aQ$ satisfying the evident strict
associativity and unit relations.

The structure of a permutative Waldhausen category on $\aC$ is
equivalent to an algebra in Waldhausen categories for the
symmetric multicategory $E\Sigma_*$ \cite[\S3]{ElmendorfMandell}, where the
unique object of $E\Sigma_*$ is taken to $\aC$.  Then $K\aC$ becomes
an $E\Sigma_{*}$-algebra in symmetric spectra; this is a particular
type of $E_{\infty}$-algebra symmetric spectrum, which is an
associative ring symmetric spectra by neglect of structure
(the symmetric multicategory of objects of $E\Sigma_{*}$ is the operad $\Sigma_{*}$ of sets).
Similarly, the
structure of a strict Waldhausen module over $\aC$ on $\aQ$ is
equivalent to specifying an algebra in Waldhausen categories for the
symmetric multicategory associated to $E \Sigma_*$ parametrizing modules, called
$E\ell\mathbf{M}^{\Sigma_{*}}$ in 
\cite[\S9.1]{ElmendorfMandell}, such that the ``ring object'' is taken
to $\aC$ and the ``module object'' to $\aQ$.  Then $K\aQ$ becomes a
$K\aC$-module in symmetric spectra.

\begin{cor}\label{cor:mods}
Let $\aC$ be a permutative Waldhausen category.  Then $K\aC$ is
naturally an $E \Sigma_*$-algebra symmetric spectrum, and in
particular an associative ring symmetric spectrum. Moreover, if
$\aD$ is a strict Waldhausen $\aC$-module, then $K\aD$ is
naturally a $K\aC$-module.
\end{cor}

Working with a permutative product has the appealing consequence that
the multicategory that arises is a familiar one, namely, the categorical
Barratt-Eccles operad $E \Sigma_*$.  However, the categories that we
work with in this paper (and that tend to arise in practice) are
symmetric monoidal categories rather than permutative categories.
This is no real limitation, since a standard construction
\cite{Isbell} rectifies any symmetric 
monoidal category into an equivalent
permutative category: The rectification of $\aC$ is
a category $\aC'$ with objects the ``words''
in the objects of $\aC$, where a word $(X_{1},X_{2},\dotsc,X_{r})$ 
corresponds to the product 
\[
\lambda(X_{1},\dotsc,X_{r})=(\dotsb (X_{1}\otimes X_{2})\otimes \dotsb)\otimes X_{r}
\]
in 
$\aC$; we associate the empty word
in $\aC'$ to the unit of the monoidal product.  
The morphisms in $\aC'$ are precisely the 
morphisms in $\aC$ between the associated products
\[
\aC'((X_{1},\dotsc,X_{r}),(Y_{1},\dotsc,Y_{s}))=
\aC(\lambda (X_{1},\dotsc,X_{r}),\lambda (Y_{1},\dotsc,Y_{s}))
\]
Concatenation provides the permutative structure.  Sending a word to
the associated product $\lambda$ defines a strong symmetric monoidal
functor $\aC'\to \aC$. The inclusion of $\aC$ in $\aC'$ as the
singleton words is also a strong symmetric monoidal functor; the
composite functor $\aC\to \aC$ is the identity, while the composite
functor $\aC' \to \aC'$ is naturally isomorphic to the identity via
the map corresponding to the identity map on the associated product.
When $\aC$ is a Waldhausen category and $\otimes$ is biexact, we use
the variant where we look at words in objects that are not $*$
together with a distinguished zero object $*$, and force a
concatenation in $\aC'$ with $*$ to result in $*$.  The resulting
category $\aC''$ becomes a Waldhausen category when we define the
cofibrations and weak equivalences to be those maps that correspond to
weak equivalences and cofibrations in $\aC$.  The functors above
remain strong symmetric monoidal equivalences, but now are exact
functors as well.

Alternatively, at the cost of complicating the multicategory in
Corollary~\ref{cor:mods}, we can work directly with symmetric monoidal
Waldhausen categories (i.e., Waldhausen categories that are symmetric
monoidal under a biexact product).  Specifying such a structure on
$\aC$ is equivalent to specifying the structure of an algebra over a
certain symmetric multicategory $\aB$ enriched in small categories,
defined as 
follows: $\Ob \aB$ is a single element.  For $k=1$, $\aB_{1}$ is the
category with one object and the identity morphism.  For $k>1$, $\aB$
is the category with objects the labelled planar binary trees with $k$
leaves, having a unique morphism between any two objects.  The
permutation action permutes the labels.  As
above, there is a symmetric multicategory parametrizing modules in
this setting; an action of $\aC$ on a Waldhausen category $\aD$
through a biexact functor endows $(\aC,\aD)$ with the structure of an
algebra over this module multicategory.  We have the following consequence:

\begin{cor}\label{cor:symmods}
Let $\aC$ be a symmetric monoidal Waldhausen category.  Then $K\aC$ is
naturally an $\aB$-algebra symmetric spectrum. Moreover, if
$\aD$ is a symmetric monoidal Waldhausen $\aC$-module, then $K\aD$ is
naturally a $K\aC$-module (parametrized by the multicategory of
$E_\infty$-modules associated to $\aB$).  
\end{cor}

%%%%%%%%%%%%%%%%%%%%%%%%%%%%%%%%%%%%%%%%%%%%%%%%%%%%%%%%%%%%%%%%%%%%%%%%

\section{Models for endomorphism $S$-algebras\\
and the double centralizer condition} 
\label{secmodels}

Classically, for a $k$-algebra $R$ and a $R$-module $M$, the double
centralizer condition for $M$ is the requirement that the natural map
\[
R\to \End_{\End_{R}(M,M)}(M,M)
\]
be an isomorphism.  Dwyer, Greenlees, and Iyengar \cite{DwyerGreenleesIyengar} 
studied the derived form of this condition.  They study the example of
$R=\splus \Omega X$ and $DX\simeq \End_{R}(S,S)$ in
\cite[\S4.22]{DwyerGreenleesIyengar}.  We review this example in this section in terms
of specific models we use in the remainder of the paper.

In our context, we are interested in the case when $X$ is a finite CW
complex.  As we will see below, Dwyer's results on convergence of the
Eilenberg-Moore spectral sequence \cite{Dwyer-EM} imply that the
double centralizer map cannot be a weak equivalence unless $X$ is
simply connected (as this is the only case in which $\pi_{1}X$ acts
nilpotently on $H_{0}(\Omega X)$).  Once we restrict to this context,
we can assume without loss of generality that $X$ is the geometric
realization of a reduced finite simplicial set.  Then we have a
topological group model $G$ for $\Omega X$ (given by the geometric
realization of the Kan loop group), and a free $G$-CW complex $P$
whose quotient by $G$ is $X$ (the twisted cartesian product
$G\subdot\times_{\tau}X\subdot$ for the universal twisting function
$\tau$; see for example Chapter VI of~\cite{May-Simplicial}).

\begin{notn}
Let $X$, $P$, and $G$ be as above.  Let $R=\splus G$, regarded as an
EKMM $S$-algebra \cite[IV.7.8]{EKMM}.  Let $SP=\splus P$, and let
$E=F_{R}(SP,SP)$.
\end{notn}

We regard $S$ as a $R$-algebra via the augmentation $R\to S$ (induced
by the map $G\to *$).  The map $SP\to S$ (induced by the map $P\to *$)
is a weak equivalence of $R$-modules.  Although $SP$ is not cofibrant,
it is semi-cofibrant \cite[1.2]{LewisMandell2}, meaning that the functor
$SP\sma_{S}(-)=P_{+}\sma(-)$ from $S$-modules to $R$-modules preserves
cofibrations and acyclic cofibrations \cite[1.3(a)]{LewisMandell2}.
Since in EKMM $S$-module categories all 
objects are fibrant, $E$ represents the correct endomorphism algebra
$\Ext_{R}(S,S)$ \cite[6.3]{LewisMandell2}.

These particular models show the strong parallel between the double
centralizer condition for $\splus \Omega X$ and the bar duality theory
of~\cite{HMS}.  The diagonal map $P\to P\times P\to X\times P$ induces
an $X$-comodule structure on $SP$
\[
SP= \splus P\to \splus (X\times P)\iso X_{+}\sma \splus P=X_{+}\sma SP.
\]
This in turn endows $SP$ with a left $DX$-module structure
\[
DX\sma_{S} SP \to DX \sma_{S}(X_{+}\sma SP) \iso (DX \sma
X_{+})\sma_{S}SP\to S\sma_{S}SP\iso SP.
\]
This left $DX$-module structure commutes with the left $R$-module
structure, and so defines a map of $S$-algebras
\[
DX\to F_{R}(SP,SP)=E.
\]
To see that this map is a weak equivalence, consider the following diagram
\[
\xymatrix{
DX\ar[rd]_{\iso}\ar[r]&F_{R}(SP,SP)\ar[d]\\
&F_{R}(SP,S),
}
\]
where the righthand map is induced by the map $SP\to S$ (induced by
the map $P\to *$), and the slanted map is the isomorphism induced by
the isomorphism $P/G=X$.  This diagram commutes since the top-right
composite is adjoint to the map $DX \sma SP\to S$ induced by the
diagonal $P\to X\times P$ followed by 
evaluation of $DX$ on $X$ and the trivial map $P\to *$, whereas the
slanted map is induced by the map $P\to X$ followed by evaluation of
$DX$ on $X$:
\[
\xymatrix{%
DX \sma P_{+}\ar[r]\ar[d]&DX\sma X_{+}\sma P_{+}\ar[d]\\
DX \sma X_{+}\ar[r]&S.
}
\]
Since the map $F_{R}(SP,SP)\to F_{R}(SP,S)$ is a weak equivalence, the
$S$-algebra map $DX\to E$ is a weak equivalence.

We can obtain a model for the map $R\to \Ext_{E}(S,S)$ as follows.
First, it is convenient to choose a cofibrant $S$-algebra
approximation $E'\to DX$.  Then the two-sided bar construction 
$SP'=B(DX,E',SP)$ is a semi-cofibrant $DX$-module approximation of
$SP$.  Furthermore, $E'$-maps $SP\to SP$ induce $DX$-maps  
$SP'\to SP'$.  By construction, the (left) action of $R$ on $SP$
commutes with the (left) action of $DX$, making $SP'$ an $R$-module in
the category of $DX$-modules, or equivalently, producing a map of
$S$-algebras $R\to F_{DX}(SP',SP')$. 

This constructs the $S$-algebra map; we need to show that this map is
a weak equivalence.  Consider the cobar construction $C\supdot(*,X,P)$,
\[
C^{n}(*,X,P)=\underbrace{X\times \dotsb \times X}_{\text{$n$ factors}}\times P,
\]
with cosimplicial maps induced from the diagonal, the inclusion of the
basepoint, and the map $P\to X$.  The inclusion of $G$ as the fiber of the
fibration $P\to X$ induces a weak equivalence $G\to \Tot
C\supdot(*,X,P)$.  Likewise, we get a map 
\[
R\to \splus \Tot C\supdot(*,X,P)\to \Tot \splus C\supdot(*,X,P).
\]
Results of Dwyer~\cite{Dwyer-EM} and
Bousfield~\cite{BousfieldCosimplicial} (for $D_{*}=\pi^{S}_{*}$) show
that this map is a weak equivalence, as $X$ is
simply-connected.  Moreover, when $X$ is not simply-connected, the
``only if'' part of Dwyer's results show that no model of this map
will be a weak equivalence.  The map $E'\to DX$ induces weak
equivalences
\[
E'\sma_{S}\dotsb \sma_{S}E'\to D(X\times \dotsb \times X).
\]
Together with the weak equivalence of $E'$-modules $SP\to S$, these
induce weak equivalences
\begin{multline*}
\splus(X\times \dotsb \times X\times P)\iso
X_{+}\sma \dotsb \sma X_{+} \sma SP\to\\
F_{S}(E'\sma_{S}\dotsb \sma_{S}E'\sma S,SP)\to
F_{S}(E'\sma_{S}\dotsb \sma_{S}E'\sma SP,SP)\\
\iso F_{DX}(DX\sma_{S}E'\sma_{S}\dotsb \sma_{S} E'\sma_{S} SP,SP).
\end{multline*}
These maps are compatible with the cosimplicial structure on the cobar
construction and the maps induced by the simplicial structure on the
bar construction $B(DX,E',SP)$, and induce a weak equivalence on
$\Tot$.  Finally, the weak equivalence of $E'$-modules $SP'\to SP$
induces a weak equivalence $F_{DX}(SP',SP')\to F_{DX}(SP',SP)$.
This describes the maps in the following diagram.
\[
\xymatrix@R-1pc{%
R\ar[r]\ar[dd]&F_{DX}(SP',SP')\ar[d]\\
&F_{DX}(B(DX,E',SP),SP)\\
\Tot \splus C\supdot(*,X,P)\ar[r]
&\Tot F_{S}(E'\sma_{S}\dotsb \sma_{S}E',SP)\ar[u]
}
\]
We have shown all maps but the top one to be weak equivalences, and so
it suffices to observe that the diagram commutes up to homotopy.  At each
cosimplicial level, the right-down composite is adjoint to the map
\[
R\sma_{S}E'\sma_{S}\dotsb \sma_{S} E'\sma_{S}SP\to SP
\]
induced by the action of $E'$ and $R$ on $SP$.  The
down-right-up composite is adjoint to the composite map
\[
R\sma_{S}E'\sma_{S}\dotsb \sma_{S} E'\sma_{S}SP\to R\sma_{S}S\to SP
\]
induced by the augmentation $E'\to S$, the weak equivalence $SP\to S$
and the inclusion of $G$ in $P$.  A contraction $P\times I\to P$ onto
the basepoint of $P$ induces a homotopy from the former map to the
latter map.

%%%%%%%%%%%%%%%%%%%%%%%%%%%%%%%%%%%%%%%%%%%%%%%%%%%%%%%%%%%%%%%%%%%%%%%%

\section{Contravariant equivalences in algebraic $K$-theory\\
and geometric {S}wan theory of spaces}\label{seccontra}

We now turn to the adjoint functors
\[
\xymatrix@C-2.75em{%
\Ext_{R}(-,S){:}\,& 
\aD_{R}\ar[rr]<.75ex>&\hspace{3.5em}
&\aD_{DX}\ar[ll]<.75ex>
&{:}\Ext_{DX}(-,S)
}%
\]
and describe our point set model for the Quillen adjunction on the
model categories $\aM_{R}$ and $\aM_{DX}$.  The easiest and most
obvious point-set model for these functors would be to use the
adjunction $F_{R}(-,SP)\colon \aM_{R}\to \aM_{DX}$ as in the previous
section, but instead we use an equivalent functor with better
multiplicative properties. 

The diagonal map $G\to G\times G$ induces a diagonal map $R\to
R\sma_{S} R$, which is clearly a map of $S$-algebras.  This endows the
category $\aM_{R}$ of $R$-modules with a symmetric monoidal product,
given by $\sma_{S}$ on the underlying $S$-modules.  As $DX$ is a
commutative $S$-algebra, the category $\aM_{DX}$ has a symmetric
monoidal product $\sma_{DX}$.  The diagonal map $SP\to SP\sma_{S}SP$
on $SP$ makes $SP$ a cocommutative coalgebra in the category of
$R$-modules and the diagonal map $SP\to SP\sma_{DX}SP$ makes $SP$ a
cocommutative coalgebra in the category of $DX$-modules.  For our
adjunctions, we need a version of $SP$ that is a commutative algebra
in both categories.

\begin{notn}
Let $\SPdual=F_{S}(SP,S)\iso S^{P_{+}}$, a left $(R\sma_{S} DX)$-module.
\end{notn}

Here $F_{S}$ (and more generally $F_{R}$ and $F_{DX}$ which we use
below) denotes the function module construction of
\cite[\S III.6.1]{EKMM}.  The commuting left $R$-module and
$DX$-module structures on $SP$ make $F_{S}(SP,S)$ naturally a right $(R\sma_{S}
DX)$-module, and we turn it into a left $(R\sma_{S} DX)$-module using
commutativity of $DX$ and the anti-involution $R\to R$ induced by the
inverse map $G\to G$.  
Using the diagonal map on $SP$, we get now a map of left $(R\sma_{S}
DX)$-modules
\[
\SPdual \sma_{DX} \SPdual \to \SPdual,
\]
which is easily seen to be associative and commutative in the appropriate sense.
%The collapse map $SP\sma_{S}
%SP\to S\sma_{S}S\iso S$ induces a map $SP\to \SPdual$ which is easily
%seen to be a left $(R\sma_{S} DX)$-module map and a weak equivalence.
% ***** No!!!!!, not a map of DX-modules *******
Moreover, we have a zigzag of weak equivalences of left $(R\sma_{S}
DX)$-modules relating $SP$ and $\SPdual$,
\[
\SPdual \from \SPdual \sma_{S}SP \to SP,
\]
where we make $\SPdual\sma_{S}SP$ a left $(R\sma DX)$-module using the
diagonal $R$-module structure and the $DX$-module structure on
$\SPdual$.  The leftward map is induced by the map of $R$-modules
$SP\to S$, and the rightward map is induced by the diagonal on $SP$
and evaluation,
\[
\SPdual \sma_{S}SP = F_{S}(SP,S) \sma_{S}SP\to
F_{S}(SP,S) \sma_{S} SP\sma_{S}SP\to S\sma_{S} SP\iso SP.
\]
This is clearly a map of $R$-modules since each map in the composite
is, and it is a map of $DX$-modules since the $DX$-module structure on
$\SPdual$ is adjoint to the map
\begin{multline*}
DX\sma_{S}\SPdual \sma_{S} SP\to
DX \sma_{S}\SPdual \sma_{S} (X_{+}\sma SP)\\
\iso
(DX \sma X_{+}) \sma_{S} (\SPdual \sma_{S} SP)\to S
\end{multline*}
induced by the diagonal on $SP$ and evaluation.

Using the commuting left $R$-module and $DX$-module structures on
$\SPdual$, we get adjoint functors
\[
\xymatrix@C-2.75em{%
F_{R}(-,\SPdual){:}\,& 
\aM_{R}\ar[rr]<.75ex>&\hspace{3.5em}
&\aM_{DX}\ar[ll]<.75ex>
&{:}F_{DX}(-,\SPdual)
}%
\]
between the (point-set) categories of $R$-modules and $DX$-modules
modeling the $\Ext_{R}(-,S)$ and $\Ext_{DX}(-,S)$ adjunction on
derived categories.
The unit maps of this adjunction are the maps
\[
X\to F_{DX}(F_{R}(X,\SPdual),\SPdual)\qquad \text{and}\qquad 
Y\to F_{R}(F_{DX}(Y,\SPdual),\SPdual)
\]
adjoint to the ($R$-module and $DX$-module) maps
\[
M\sma_{S} F_{DX}(M,\SPdual)\to \SPdual\qquad \text{and}\qquad
N\sma_{S} F_{R}(N,\SPdual)\to \SPdual
\]
induced by evaluation.  Since fibrations and weak equivalences in EKMM
module categories are detected on the underlying $S$-modules, the functors
$F_{R}(-,\SPdual)$ and $F_{DX}(-,\SPdual)$ convert cofibrations and acyclic
cofibrations to fibrations and acyclic fibrations.  The adjunction
above is therefore a Quillen adjunction.  

We note that these functor are lax symmetric monoidal.    
We have the natural transformations
\begin{align*}
F_{R}(M_{1},\SPdual)\sma_{DX}F_{R}(M_{2},\SPdual)&\to
F_{R}(M_{1}\sma_{S} M_{2},\SPdual\sma_{DX}\SPdual)\\
&\to F_{R}(M_{1}\sma_{S} M_{2},\SPdual)\\
\intertext{and}
F_{DX}(N_{1},\SPdual) \sma_{S} F_{DX}(N_{2},\SPdual)&\to
F_{DX}(N_{1}\sma_{DX}N_{2}, \SPdual\sma_{DX}\SPdual)\\
&\to
F_{DX}(N_{1}\sma_{DX}N_{2}, \SPdual)
\end{align*}
induced by the multiplication $\SPdual \sma_{DX} \SPdual\to \SPdual$,
which is both a map of $R$-modules and of $DX$-modules.  Note that
because $G$ is a CW complex, $M_{1}\sma_{S} M_{2}$ is in fact a
cofibrant $R$-module when $M_{1}$ and $M_{2}$ are cofibrant
$R$-modules.  The lax unit natural transformations
\[
DX\to F_{R}(S,\SPdual)\iso F_{S}(SP\sma_{R} S,S)
\qquad\text{and}\qquad
S\to F_{DX}(DX,\SPdual)\iso \SPdual
\]
are induced by the identification $P/G=X$ (for the first map) and the
map of $R$-modules $SP\to S$ (for the second map).

When $M$ is a cofibrant $R$-module approximation to $SP$ (for example,
$X=SP\sma S_{c}$ for a cofibrant $S$-module approximation of $S$), we
have a weak equivalence of $DX$-modules, 
\[
DX\to E=F_{R}(SP,SP)\to F_{R}(M,SP)\simeq F_{R}(M,\SPdual).
\]
It follows that the left derived functors of
$F_{R}(-,\SPdual)$ and $F_{DX}(-,\SPdual)$ induce an equivalence
between the thick 
subcategories of the homotopy categories generated by $S$ in $\aD_{R}$
and by $DX$ in $\aD_{DX}$.  The latter is the category of compact
objects $\aD^{c}_{DX}$.  As in the introduction, we denote the former
subcategory by $\thick{R}$.

\begin{prop}
The derived functors $\Ext_{R}(-,S)$ and $\Ext_{DX}(-,S)$ induce
inverse equivalences between $\thick{R}$ and $\aD_{DX}^{c}$.
\end{prop}

Likewise, when $N$ is a cofibrant $DX$-module approximation to $SP'$,
we have a weak equivalence of $R$-modules 
\[
R\to F_{DX}(SP',SP')\to F_{DX}(SP',SP) \simeq F_{DX}(SP',\SPdual)\to F_{DX}(N,\SPdual).
\]
It follows that 
the left derived functors of
$F_{R}(-,\SPdual)$ and $F_{DX}(-,\SPdual)$ induce an equivalence between the thick
subcategories of the homotopy categories generated by $S$ in $\aD_{DX}$
and by $R$ in $\aD_{R}$.  The latter is the category of compact
objects $\aD^{c}_{R}$.  As in the introduction, we denote the former
subcategory by $\thick{DX}$.

\begin{prop}
The derived functors $\Ext_{R}(-,S)$ and $\Ext_{DX}(-,S)$ induce
inverse equivalences between $\aD_{R}^{c}$ and $\thick{DX}$.
\end{prop}

We obtain Waldhausen category structures modeling each of the
subcategories $\aD^{c}_{R}$, $\thick{R}$ in $\aD_{R}$ and
$\aD^{c}_{DX}$, $\thick{DX}$ in $\aD_{R}$ as follows.  We consider the
full subcategory of cofibrant objects in the model category of
$R$-modules or $DX$-modules whose image in the homotopy category lies
in the subcategory.  (To make these categories small, we can fix a set
$\mathfrak{X}$ of sufficiently large cardinality and restrict to objects
whose point-sets are subsets of $\mathfrak{X}$ as in
\cite[1.7]{BlumbergMandell}.)  We denote these Waldhausen categories
as $\aM^{c}_{R}$, $\model{R}$, $\aM^{c}_{DX}$, and $\model{DX}$,
respectively.  We then get associated $K$-theory spectra, including
Waldhausen's algebraic $K$-theory of $X$ and the geometric Swan theory
of $X$.

\begin{defn}
$A(X)=K(R)=K(\aM^{c}_{R})$, $G(X)=K(\model{R})$, $K(DX)=K(\aM^{c}_{DX})$.
\end{defn}

The biexact smash product $\sma_{S}$ makes $G(X)$ into an
$E_{\infty}$ ring symmetric spectrum and the biexact smash product
$\sma_{DX}$ makes $K(DX)$ into an $E_{\infty}$ ring symmetric
spectrum by Theorem~\ref{thm:multstruct}.  Likewise, the biexact
functors $\sma_{S}$ and $\sma_{DX}$ 
make $A(X)$ into a module over $G(X)$ and $K(\model{DX})$ into a
module over $K(DX)$.  We next explain how the functors $F_{R}(-,\SPdual)$
and $F_{DX}(-,\SPdual)$ induce weak equivalences of these $E_{\infty}$ ring
symmetric spectra and modules.

Although the functors $F_{R}(-,\SPdual)$ and $F_{DX}(-,\SPdual)$ are not exact
(and do not land in the model Waldhausen categories), we do
immediately obtain weak equivalences of spectra $K(DX)\to G(X)$ and
$K(\model{DX})\to A(X)$, using the $\Spdot$ construction, a 
homotopical variant of the $\Sdot$ construction introduced
in~\cite{BlumbergMandell}.  Rather than working with pushouts over
cofibrations, the $\Spdot$ construction depends on a theory of
``homotopy cocartesian'' squares.  The $\Spdot$ construction replaces
the cofibrations and pushouts in $\Sdot$ with homotopy cocartesian
squares.  Under mild hypotheses~\cite[App.~A]{BlumbergMandell2}, the
natural inclusion $\Sdot \aC \to \Spdot \aC$ is a weak equivalence.
To study the products and pairings, we take a different approach that
allows us to continue working only with objects that are cofibrant.

First consider the categories $\model{R}$ and $\aM^{c}_{DX}$.  For each
$q$, $n_{1},\dotsc, n_{q}$, consider the category whose objects consist of an
element $A$ of $\Sdotq[n_{1},\dotsc,n_{q}]\model{R}$, an element $B$
of $\Sdotq[n_{1},\dotsc,n_{q}]\aM^{c}_{DX}$ and weak equivalences
\[
\phi_{i_{1},j_{1};\dotsb;i_{q},j_{q}}\colon 
A_{i_{1},j_{1};\dotsb;i_{q},j_{q}}\to
F_{DX}(B_{n_{1}-j_{1},n_{1}-i_{1};\dotsb;n_{q}-j_{q},n_{q}-i_{q}},\SPdual)
\]
making the $\Ar[n_{1}]\times \dotsb \times \Ar[n_{q}]$ diagram
commute.  A map $(A,B,\phi)$ to $(A',B',\phi')$ consists of weak
equivalences $A\to A'$, $B\to B'$ 
%in $\Sdotq[n_{1},\dotsc,n_{q}]
%\model{R}$ and $\Sdotq[n_{1},\dotsc,n_{q}] \aM^{c}_{DX}$ respectively
such that the composite 
\[
A_{*}\to A'_{*}\overto{\phi'}F_{DX}(B'_{*},\SPdual)\to F_{DX}(B_{*},\SPdual)
\]
is $\phi_{*}$.  This forms a multi-simplicial category, where we use
the opposite ordering in each simplicial direction on
$\Sdotq\aM^{c}_{DX}$.  Taking the classifying space, we obtain a
sequence of spaces $T(q)$ with the structure of a symmetric spectrum.

The smash products on $\aM_{R}$ and $\aM_{DX}$ and lax symmetric
monoidal transformations above induce maps
\[
T(p)\sma T(q)\to T(p+q)
\]
and a multiplication $T\sma T\to T$.  We obtain a map $T\to G(X)$ 
dropping the $\aM^{c}_{DX}$ data; we also obtain a map $T\to
K(DX)$ by dropping the $\model{R}$ data and using the
canonical homeomorphism between the geometric realization of a
simplicial set and its opposite.  Both maps preserve the $E_{\infty}$
structures; Lemmas~\ref{lemquila} and~\ref{lemcompSdot} below complete
the proof of Theorem~A and 
the first part of Theorem~C by showing that these maps are are weak
equivalences.   

The analogous construction, with $\aM^{c}_{R}$ and $\model{DX}$ in
place of $\model{R}$ and $\aM^{c}_{DX}$, produces a symmetric
spectrum $U$ that is a module over $T$.  The analogous maps $U\to
A(X)$ and $U\to K(\model{DX})$ are $T$-modules maps; again
Lemmas~\ref{lemquila} and~\ref{lemcompSdot} show that these maps are
weak equivalences and 
complete the proof of Theorem~B and the remaining part of Theorem~C.

Before stating Lemma~\ref{lemquila}, we abstract the construction used
to build the pieces of $T$ and $U$.  Consider the following construction.

\begin{cons}\label{consmiddle}
Let $\aC$ be a Waldhausen category, $\baM$ be a pointed closed model category
and let $\aM$ be a closed Waldhausen subcategory of cofibrant objects in
$\baM$, i.e., a Waldhausen category under the cofibrations and weak
equivalences from $\baM$, which is closed under weak equivalences in
$\baM$.  Let 
$F\colon \aC\to \baM$ be a contravariant functor that takes $*$ to
$*$, cofibrations to fibrations, and weak
equivalences to weak equivalences.  Define $MF$ to be the following
category.  An object of $MF$ consists of an object $A$ of $\aM$, an
object $B$ of $\aC$ and a weak equivalence $\phi \colon A\to FB$.  A map in $MF$
from $(A,B,\phi)$ to $(A',B',\phi')$ consists of weak equivalences
$A\to A'$ and $B\to B'$ such that the composite map
\[
A\to A'\overto{\phi'}FB'\to FB
\]
is $\phi$.  We have canonical functors $MF\to \w\aC$ and $MF\to \w\aM$
obtained by dropping the $\baM$ and $\aC$ data, respectively.
\end{cons}

\begin{lem}\label{lemquila}
With notation as above:
\begin{enumerate}
\item If every object $A$ of $\aC$,
$FA$ is weakly equivalent in $\baM$ to an object of $\aM$, then the
functor $MF\to \w\aC$ induces a weak 
equivalence on nerves. 
\item If $\aC$ is a closed Waldhausen subcategory of
cofibrant objects in a closed model category $\baC$ and $F$ is a left
Quillen adjoint that induces an equivalence between the full
subcategories of $\Ho\baC$ and $\Ho\baM$ generated by $\aC$ and $\aM$,
then $MF\to \w\aM$ also induces a weak equivalence on nerves.
\end{enumerate}
\end{lem}

\begin{proof}
For the first statement, we apply Quillen's Theorem~A.  
For an object $B$ of $\aC$, the relevant category
$FM\downarrow B$ has
objects the maps $\phi \colon A\to FC$, $\gamma \colon C\to B$ where
$A$ is a cofibrant object in $\baM$, $C$ is an object in $\aC$, and
$\phi$ and $\gamma$ are weak equivalences.  The nerve of this
category is equivalent to the nerve of the subcategory where $C=B$ and
$\gamma$ is the identity.  This is the category of cofibrant
approximations of the fibrant object $FC$; work of Dwyer-Kan
(cf.~\cite[6.12]{DKModel}) shows that the nerve of this category is
contractible.

For the second statement, let $G$ denote the contravariant left
adjoint of $F$.  Then under the hypotheses of the second statement, a
map $A\to FB$ is a weak equivalence if and only if the adjoint map
$B\to GA$ is a weak equivalence.  The second statement now follows
from the first.
\end{proof}

In the case considered above, we are looking at functors $F$ of the
form
\[
\Sdotq\model{DX}\to \Ar[\seqdot](\aM_{R})
\quad \text{or} \quad 
\Sdotq\aM^{c}_{DX}\to \Ar[\seqdot](\aM_{R}),
\]
where we have written $\Ar[\seqdot](\aC)$ for the category of functors
from $\Ar[\seqdot]$ to $\aC$ (where $\Ar[n_{1},\dotsc,n_{q}]$ is as in
Construction~\ref{consitsdot}).
Both $\Sdotq\model{DX}$ and $\Sdotq\aM^{c}_{DX}$ are closed Waldhausen
subcategories of the cofibrant objects in\break
$\Ar[\seqdot](\aM_{DX})$.  Since every map in $\Ar[\seqdot](\aM_{R})$
is weakly equivalent to a cofibration, and a commuting square in
$\aM_{R}$ is a homotopy pushout square if and only if it is a homotopy
pullback square if and only if it is weakly equivalent to a pullback
square of fibrations, an easy inductive argument proves the following
lemma. 

\begin{lem}\label{lemcompSdot}
The functor $F_{DX}(-,\SPdual)$ induces equivalences between:
\begin{enumerate}
\item The full subcategory of the homotopy category of
$\Ar[n_{1},\dotsc,n_{q}](\aM_{DX})$ generated by objects of
$\Sdotq[n_{1},\dotsc,n_{q}]\aM^{c}_{DX}$, and
\item the full subcategory of the homotopy category of
$\Ar[n_{1},\dotsc,n_{q}](\aM_{R})$ generated by objects of
$\Sdotq[n_{1},\dotsc,n_{q}]\model{R}$.
\end{enumerate}
It also induces equivalences between:
\begin{enumerate}
\item The full subcategory of the homotopy category of
$\Ar[n_{1},\dotsc,n_{q}](\aM_{DX})$ generated by objects of
$\Sdotq[n_{1},\dotsc,n_{q}]\model{DX}$, and 
\item the full subcategory of the homotopy category of
$\Ar[n_{1},\dotsc,n_{q}](\aM_{R})$ generated by objects of
$\Sdotq[n_{1},\dotsc,n_{q}]\aM^{c}_{R}$.
\end{enumerate}
\end{lem}

%%%%%%%%%%%%%%%%%%%%%%%%%%%%%%%%%%%%%%%%%%%%%%%%%%%%%%%%%%%%%%%%%%%%%%%%
\section{Covariant equivalences in algebraic $K$-Theory\\
and geometric {S}wan theory of spaces}\label{secinvolution}

Using the models described in Section~\ref{secmodels}, the generalized
Morita theory of~\cite{DwyerGreenleesIyengar} admits a point-set refinement into adjoint
pairs of covariant functors
\[
\xymatrix@C-2.75em{%
F_{R}(SP,-){:}\,& 
\aM_{R}\ar[rr]<-.75ex>&\hspace{3.5em}
&\aM_{DX}\ar[ll]<-.75ex>
&{:}(-) \sma_{DX} SP
}%
\]
and
\[
\xymatrix@C-2.75em{%
(-) \sma_{R} SP' {:}\,& 
\aM_{R}\ar[rr]<.75ex>&\hspace{3.5em}
&\aM_{DX}\ar[ll]<.75ex>
&{:} F_{DX}(SP',-),
}%
\]
forming Quillen adjunctions.  Here we switch between left and
right modules at will using the commutativity of $DX$ and the
anti-involution on $R$ (induced by the inverse map on the topological
group $G$).

Since $S$ is compact in $\aD_{R}$, the first
adjunction induces an equivalence between the localizing subcategory
of $\aD_{R}$ generated by $S$ and $\aD_{DX}$, and, in particular, it
restricts to an equivalence between $\thick{R}$ and $\aD^{c}_{DX}$.
In general, $S$ is not compact in $\aD_{DX}$, but nonetheless the
second adjoint pair yields an equivalence between $\thick{DX}$ and
$\aD^{c}_{R}$.  In this case, one of the functors in each pair is
exact, and so Waldhausen's approximation theorem (or the more general
formulations of \cite{BlumbergMandell2} or
\cite{ToenVezzosi}) implies that these equivalences induce equivalences
\[
K(DX) \to G(X)\qquad \text{and}\qquad 
A(X) \to K(\model{DX}).
\]

Combining these equivalences with the equivalences of the previous
section, we obtain self-homotopy equivalences on $A(X)$ and $G(X)$.
We complete our analysis by identifying these as the standard
Spanier-Whitehead duality involution on $A(X)$ and an analogous
involution on $G(X)$. (Note that when $X$ is a smooth manifold, this
involution is generally not compatible with the involution on
pseudo-isotopy theory unless $X$ is parallelizable \cite{Vogell1126}.)
Roughly, the involution on $A(X)$ is 
given by the functor which takes a left $R$-module $M$ to the right
$R$-module $\Ext_{R}(M,R)$, which we transform into a left $R$-module
via the anti-involution $R\to R^{\op}$.   For $G(X)$,
the involution is similar but with $\Ext_{S}(M,S)$ instead.

Because the duality maps are contravariant, it is convenient to work
with Waldhausen categories modeling the opposite categories of
$\aD_{R}^{c}$ and $\thick{R}$.  As observed in
\cite[\S1]{BlumbergMandell}, $\aM^{\op}_{R}$ has the structure of a
Waldhausen category with weak equivalences the maps opposite to the
usual weak equivalences and cofibrations the maps opposite to the Hurewicz
fibrations.  Let $\aM^{\op,c}_{R}$ be the full subcategory of objects
that are opposite to compact objects in $\aD^{c}_{R}$, and let
$\opmodel{R}$ be the full subcategory of $\aM^{\op}_{R}$ opposite to
objects in $\thick{R}$ (again, we can make these latter two Waldhausen
categories small by restricting to subsets of a set with high
cardinality).  The argument for \cite[1.1]{BlumbergMandell} (see
discussion following \cite[2.9]{BlumbergMandell}) provides weak
equivalences
\[
A(X)=K(\aM^{c}_{R})\simeq K(\aM^{\op,c}_{R})
\qquad \text{and}\qquad
G(X)=K(\model{R}) \simeq K(\opmodel{R}).
\]
Essentially the map on $\Sdot[n]$ sends $A=\{A_{i,j}\}$ to
$A'=\{A'_{i,j}\}$ where $A'_{i,j}\simeq A_{n-j,n-i}$ and the
pushouts over cofibrations have been replaced by equivalent pullbacks
over fibrations.

The functors $F_{R}(-,R)\colon \aM^{c}_{R}\to \aM^{\op,c}_{R}$
and $F_{S}(-,S)\colon \model{R}\to \opmodel{R}$ are then exact.
Under the equivalences above, the induced maps on $K$-theory represent
the canonical involution.  Thus, it now suffices to compare our
composite functors to these functors. 

In the case of $A(X)$, the composite of our equivalences is the functor
$\aM^{c}_{R} \to \aM^{\op,c}_{R}$
defined as
\[
M \mapsto F_{DX}(M \sma_{R} SP', \SPdual).
\]
By adjunction, this is naturally isomorphic to $F_{R}(-,
F_{DX}(SP',\SPdual))$.  The weak equivalence $R\to F_{DX}(SP',\SPdual)$ then
induces a natural weak equivalence from the duality functor $F_{R}(-,R)$.

For $G(X)$, the argument above shows that the composite map on
$K(DX)\to K(\aM^{\op,c}_{DX})$
is the functor $F_{R}(-\sma_{DX}SP,\SPdual)$ and is naturally weakly
equivalent to the duality map $F_{DX}(-,DX)$.  On 
the other hand 
$F_{R}(SP,-)\colon \opmodel{R}\to \aM^{\op,c}_{DX}$ is exact and the
following solid arrow diagram commutes up to natural isomorphism.
\[
\xymatrix@R+1.5pc@C+3pc{%
\model{R}\ar@{.>}[r]^{F_{S}(-,S)}
&\opmodel{R}\ar[d]^{F_{R}(SP,-)}\\
\aM^{c}_{DX}\ar[ur]^{F_{DX}(-,\SPdual)\ }\ar[r]_{F_{R}(-\sma_{DX}SP,\SPdual)}
\ar@{.>}[u]^{(-)\sma_{DX}SP}
&\aM^{\op,c}_{DX}
}
\]
The composite of the dotted arrows is the functor
$F_{S}((-)\sma_{DX}SP,S)$.  By the smash-function adjunction, we see
that this functor is naturally isomorphic to 
$F_{DX}(-,F_{S}(SP,S))$, which is the diagonal arrow since
$\SPdual=F_{S}(SP,S)$. 

%%%%%%%%%%%%%%%%%%%%%%%%%%%%%%%%%%%%%%%%%%%%%%%%%%%%%%%%%%%%%%%%%%%%%%%%
% Bibliography
%%%%%%%%%%%%%%%%%%%%%%%%%%%%%%%%%%%%%%%%%%%%%%%%%%%%%%%%%%%%%%%%%%%%%%%%

\bibliographystyle{plain}

\end{document}